\documentclass[11p,reqno]{amsart}
\usepackage{tikz-cd}
\usepackage{amsmath}
\usepackage{amsfonts}
\usepackage{amssymb}
\usepackage{graphicx}
\usepackage{color}
\newtheorem{theorem}{Theorem}[section]
\newtheorem*{theorem*}{Theorem}
\newtheorem{lemma}{Lemma}[section]

\newtheorem{proposition}{Proposition}[section]

\numberwithin{equation}{section}

\begin{document}
	\title[Morse]{Dynamics of the Morse vector field}

\author{Yijian Zhang}

\address{Yijian Zhang. School of Mathematical Sciences,  University of Science and Technology of China, Hefei, Anhui, China, 230026}

\email{zyj\_math@mail.ustc.edu.cn}

\begin{abstract}
The (negative) gradient vector fields of Morse functions on a compact manifold provide an important example in dynamical system. In this note we prove two important properties of this kind of vector field: Connectedness of critical points through orbits and exponential shrinkage of the flow on stable submanifolds. We also find applications in showing some vanishing results for maps or curvature operators.  
\end{abstract}

\maketitle

\section{Introduction}

 Morse theory is a powerful tool in differential topology that connects the critical point structure of functions on manifolds to the topology of manifolds. The key to Morse theory is the link between Morse functions and their gradient vector fields which serve as a good model in dynamical system: these vector fields have no periodic orbits and they give cell decomposition of manifolds. In this paper we discuss these points in detail.   

 A Morse function on a Riemannian manifold $M$ is a smooth function $f:M\rightarrow{\mathbb{R}}$, where all critical points are non-degenerate in the sense of Hessian of $f$. These non-degenerate critical points are isolated (hence finite by compactness of $M$) and classified by their indices $\mathrm{ind}_f$, the number of negative eigenvalues of the $\mathrm{Hess}(f)$. 
 Sticking to the convention (e.g.\cite{Jo,Lau,Sch}), we consider the negative gradient flow
 \begin{equation}
  \begin{cases}
    \Phi : M\times\mathbb{R}\rightarrow\mathbb{R} 
    \\
    \frac{\partial}{\partial t}\Phi(x,t) = -\nabla f(\Phi(x,t))
    \\
    \Phi(0,t)= \mathrm{id}_M
  \end{cases}
\end{equation}
   We denote this time-dependent diffeomorphism $\Phi(\cdot,t)$ by $\phi_t$. Since the compactness of $M$ forces the flow line away from critical points and towards critical points, the gradient flow has no periodic orbits, and induces the concept of stable and unstable manifolds:
 \begin{equation}
  \begin{aligned}
    W^u(p) &= \{x\in M \mid \lim_{t\to +\infty} \phi_t(x) = p\}, \\
    W^s(p) &= \{x\in M \mid \lim_{t\to -\infty} \phi_t(x) = p\}
  \end{aligned}
\end{equation}
,which are diffeomorphic to Euclidean spaces of dimension $\mathrm{ind}_f(p)$ and $\mathrm{dim}(M)-\mathrm{ind}_f(p)$.
% They are indeed injective immersed submanifolds of $M$ due to the famous stable manifold theorem!

  Clearly $M$ is the disjoint union of all stable (resp. unstable) manifolds. The key observation is, the stable manifolds of local minimum (i.e. $p\in \mathrm{Crit}(f)$ with $ind_f(p)=0$) cover the most of $M$, which is an open dense subset since the others have codimension at least 1. Our main results and applications rely heavily on this fact. 

  In addition, the stable manifold decomposition can be an actual cell decomposition which makes $M$ a CW-complex, if a transversality condition known as the Morse-Smale condition (cf. background in \cite{Sch}) is satisfied. Notice that this decomposition is different from the one provided by the ascending level sets (\cite{Mi}), and produces the famous Morse homology (\cite{Lau}). Although the Morse-Smale condition is generic for choices of Riemannian metrics on $M$, it is not needed in our paper.

\section{Main results}

  From now on, $M$ denotes a connected, compact Riemannian manifold. We begin with the connectedness of critical points via flow lines:

\begin{theorem}\label{connectedness}
    Any two critical points p,q of a Morse function f on M can be connected by multiple orbits of negative gradient flow. Precisely, there exists a sequence of critical points $p_1,\cdots,p_n$ with $p_1=p$, $p_n=q$, such that either $p_i$ flows to $p_{i+1}$, or $p_{i+1}$ flows to $p_i$.
\end{theorem}
\begin{proof}
    Let $q_1,\cdots,q_N$ be all local minima of f, we already know that $\coprod_{i=1}^N W^s(q_i)$ is open dense in M, and $\forall x\in W^s(q_i)$ is contained in a flow line towards $q_i$. Consider the closure $\overline{W^s(q_i)}$, we have $\cup_{i=1}^N\overline{W^s(q_i)}=\overline{\cup_{i=1}^N W^s(q_i)}=M$. We propose the following lemma:
\begin{lemma}\label{flow}
    Every point in $\overline{W^s(q_i)}$ lies in a piecewise flow line towards $q_i$.
\end{lemma}
    With lemma 2.1, it's easy to see that if $\overline{W^s(q_i)}\cap \overline{W^s(q_j)}\neq \emptyset$, then $q_i$ and $q_j$ is connected. Now we divide these N minima into k maximal parts: $\{q_1,\cdots,q_N\}=\cup_{j=1}^k E_j$, such that the minima in the same $E_j$ can be connected to each other. Notice that for $j\neq k$, the intersection of $\cup_{i\in E_j}\overline{W^s(q_i)}$ and $\cup_{i\in E_k}\overline{W^s(q_i)}$ must be empty, otherwise minima in $E_j$ and $E_k$ are connectable, which contradicts the maximality of $E_j$. Thus $M$ is the disjoint union of closed subsets $\cup_{i\in E_j}\overline{W^s(q_i)}$, $j=1,\cdots,k$. Since $M$ is connected, $k=1$ and we obtain the connectedness of all local minima, hence connectedness of all critical points, as all points flow down to a minimum. \end{proof}

\begin{proof}[Proof of {Lemma}~\ref{flow}]
    $\forall p\in \overline{W^s(q_1)}$, there exists a sequence $\{x_i \}_{i=1}^\infty$ in ${W^s(q_1)}$ that converges to $p$. By continuity of $\phi_t$, $\phi_t(x_i)$ in converges to $\phi_t(p)$, and by $\phi_t(x_i)\in W^s(q_1)$ we get $\phi_t(p)\in \overline{W^s(q_1)}$. If $p$ isn't critical, then $\phi_t(p)$ naturally flows down to some critical point which is still contained in $\overline{W^s(q_1)}$. Thus it remains to show that if $p$ is a critical point, there exists a flow line starting from $p$ which lies in $\overline{W^s(q_1)}$.

    Firstly, we fix small $r$ such that the distance between any two critical points with respect to the given metric is at least $2r$, then $|\nabla f|$ has a positive lower bound $c$ outside $\cup\{B(q,\frac{r}{2})|q\in \mathrm{Crit}(f)\}$. Let $\epsilon$, $\delta$ be small constants to be determined. The key is that outside $\cup\{B(q,\frac{r}{2})|q\in \mathrm{Crit}(f)\}$, the difference of Morse function $f$ controls the length of the flow line:
\begin{equation*}
    \begin{aligned}
        f(\phi_{a}(x))-f(\phi_{b}(x))=-\int_{a}^{b}\frac{\mathrm{d}}{\mathrm{d}t}f(\phi_t(x))dt \ \ \ \ 
    \\    = \int_{a}^{b}|\nabla f(\phi_t(x))|^2dt  
      \geq c\int_{a}^{b}|\nabla f(\phi_t(x))|dt  \ \ \ \ \ 
    \\ = c\int_{a}^{b}|\frac{\mathrm{d}}{\mathrm{d}t}\phi_t(x)|dt = c \cdot \text{Length}(\phi_t(x):t\in [a, b])
    \end{aligned}
\end{equation*}
    
    Thus we choose $\epsilon<\frac{cr}{4}$ such that $[f(p)-\epsilon,f(p))$ doesn't contain critical value of $f$ and $\delta<\frac{r}{2}$ such that $\forall x\in B(p,\delta)$, $|f(x)-f(p)|<\epsilon$. Assume $x_i\in B(p,\delta)$ without loss of generality. We assert that for time $t_i>0$ with $f(\phi_{t_i}(x_i))=f(p)-\epsilon$, the flow line segment $\phi_s(x_i)$, $s\in [0,t_i]$ must lie in $B(p,r)$. If $\phi_s(x_i)\notin B(p,r)$, by selecting biggest time $\tilde{t_i}>0$ such that $\phi_{\tilde{t_i}}(x_i)\in \partial B(q,\frac{r}{2})$, we have
\begin{equation}\label{ineq}
    \begin{aligned}
         2\epsilon> f(x_i)-f(\phi_{t_i}(x_i))> f(\phi_{\tilde{t_i}}(x_i))-f(\phi_{s}(x_i)) 
      \\   \geq c \cdot \text{Length}(\phi_t(x_i):t\in [\tilde{t_i}, s]), \ \ \ \ \ \ \ \ \ \ \   
    \end{aligned}
\end{equation}
if $\phi_{t_i}(x_i)\notin B(p,r)$, then $\text{Length}(\phi_t(x_i):t\in [\tilde{t_i}, t_i])\geq \frac{r}{2}$, which contradicts the choice of $\epsilon$. 

  Another important fact is that $t_i\rightarrow +\infty$ as $i\rightarrow +\infty$. If not, after passing to a subsequence, we may assume that $t_i\leq C$. However, $\phi_C(x_i)$ also converges to $\phi_C(p)=p$, hence 
\begin{equation*}
    f(p)=\lim_{i\to +\infty}f(\phi_C(x_i))\leq \lim_{i\to +\infty}f(\phi_{t_i}(x_i))=f(p)-\epsilon,
\end{equation*}
a contradiction.

  The compactness of $M$ now guarantees a limit $y\in f^{-1}(f(p)-\epsilon)\cap \overline{W^s(q_1)}$ for a subsequence of $\{\phi_{t_i}(x_i)\}$, still denoted by $\{\phi_{t_i}(x_i)\}$. Suppose $\lim_{t\to -\infty} \phi_t(y)=p'$, then there exists $T>0$ such that the distance $d(\phi_{-T}(y),p')<\frac{r}{2}$. Since $\phi_{t_i-T}(x_i)\to \phi_{-T}(y)$, for large enough $i$, we have $d(\phi_{t_i-T}(x_i),p')<r$. However, for $t_i\geq T$, we also have $d(\phi_{t_i-T}(x_i),p)<r$ in view of the previous assertion. Therefore $d(p,p')<2r$, which implies $p'=p$ by the choice of $r$. Recall that $y\in f^{-1}(f(p)-\epsilon)$ isn't critical, thus we obtain the desired flow line $\phi_t(y)$ which starts from $p$   and lies in $\overline{W^s(q_1)}$.
    \end{proof}\vspace{-1\baselineskip}

  Next, we show the exponential shrinkage of the negative gradient flow on stable manifolds. Though it's a classical result in the literature which can be generalized to hyperbolic systems (see Section 8.3 of \cite{Jo}), we present here a direct, coordinate-free proof via pull-back bundle computation:

\begin{theorem}\label{shrinkage}
    For every $x\in W^s(p)$ and $v_x\in T_xW^s(p)$, there exists $c>0$ such that $\big| \mathrm{d}\phi_t|_x(v_x)\big|=O(e^{-ct})$ as $t \to +\infty$.
\end{theorem}\vspace{-2\baselineskip}

\begin{proof}[Preparation for the proof]
    In order to take covariant derivatives of the cross section $\mathrm{d}\phi_t(v)=\mathrm{d}\Phi|_{(\cdot~, t)}(v,0)$, we introduce the pull-back bundle $\Phi^{-1}(TM)$ of $TM$ and the pull-back connection $\tilde\nabla$ of the Levi-Civita connection $\nabla$ on $TM$. The bundle $\Phi^{-1}(TM) \to M$ admits a natural definition for which there exists a $C^{\infty}$-linear morphism $\psi:\Phi^{-1}(TM) \to TM$ such that the diagram 
\[
    \begin{tikzcd}[row sep=3em, column sep=3em]
\Phi^{-1}(TM)  \arrow[r, "\psi"] \arrow[d,"\tilde\pi"'] & TM \arrow[d,"\pi"] \\
M\times \mathbb{R} \arrow[r,"\Phi"] & M \\
    \end{tikzcd}\vspace{-2.5\baselineskip}
\]  
commutes and $\psi_{(x,t)}:\Phi^{-1}(TM)_{(x,t)} \to TM_{\Phi(x,t)}$ is an isomorphism for every $(x,t)\in M\times \mathbb{R}$. 
  
  Given any section $s\in \Gamma(M,TM)$, the pull-back section $\Phi^*s\in \Gamma(M\times \mathbb{R}, \Phi^{-1}(TM))$ is defined by $\Phi^*s_{(x,t)}=\psi^{-1}_{(x,t)}s_{\Phi{(x,t)}}$. In this way, the pull-back connection $\tilde\nabla$ can be given by $\tilde\nabla_X\Phi^*s=\Phi^*(\nabla_{\mathrm{d}\Phi(X)}s)$. Note that $\tilde\nabla$ preserves the pull-back metric $\Phi^* g$, which is induced by $g$ via the isomorphism $\psi$. 
    
  Now $\mathrm{d}\Phi|_{(x, t)}(v,0)$ can be viewed as a section of $\Phi^{-1}(TM)$, and we could carry out calculation which is the same as computing the first variation in harmonic maps. For details on the latter, we refer to \cite{Ura}.
  \end{proof}
\begin{proof}
  Let $e(x,t)=\frac{1}{2}\big| \mathrm{d}\phi_t|_x(v_x)\big|^2_{\Phi^* g}=\frac{1}{2}\big|\mathrm{d}\Phi|_{(x, t)}(v,0)\big|_{\Phi^* g}^2\in C^{\infty}(M\times \mathbb{R},\mathbb{R})$, and $E(t)=e(x,t)$ for fixed $x$. By direct calculation
\begin{equation}\label{ODE}
 \begin{aligned}
   \frac{\mathrm{d}}{\mathrm{d}t}E(t)= &(0,\frac{\partial}{\partial t})e(x,t)  \\ =&\frac{1}{2}(0,\frac{\partial}{\partial t}) \big\langle\mathrm{d}\Phi|_{(x, t)}(v,0),\mathrm{d}\Phi|_{(x, t)}(v,0)\big\rangle_{\Phi^* g}  
   \\  =&\big\langle \tilde\nabla_{(0,\frac{\partial}{\partial t})} \mathrm{d}\Phi(v,0)|_{(x, t)},\mathrm{d}\Phi|_{(x, t)}(v,0)\big\rangle_{\Phi^* g}  \\ \textbf{=}& \big\langle \tilde\nabla_{(v,0)} \mathrm{d}\Phi(0,\frac{\partial}{\partial t})|_{(x, t)},\mathrm{d}\Phi|_{(x, t)}(v,0)\big\rangle_{\Phi^* g}
  \\  =&\big\langle \tilde\nabla_{(v,0)} \Phi^*(-\nabla f)|_{(x, t)},\mathrm{d}\Phi|_{(x, t)}(v,0)\big\rangle_{\Phi^* g}  
  \\ =& -\big\langle \nabla_{\mathrm{d}\Phi|_{(x, t)}{(v,0)}} \nabla f,\mathrm{d}\Phi|_{(x, t)}(v,0)\big\rangle_g
   \\  =& -\text{Hess}(f)|_{\phi_t(x)}(\mathrm{d}\phi_t|_x(v_x),\mathrm{d}\phi_t|_x(v_x))
 \end{aligned}
\end{equation}
, where the fourth equality follows from the torsion-freeness of $\nabla$:
\begin{equation*}
    (\mathrm{d}_{\tilde\nabla} \mathrm{d}\Phi)(X,Y)= \tilde\nabla_X \mathrm{d}\Phi(Y)-\tilde\nabla_Y \mathrm{d}\Phi(X)-\mathrm{d}\Phi([X,Y])=0.
\end{equation*}
Its proof can be found on p. 129 of \cite{Ura}.

 Notice that $\mathrm{d}\phi_t|_x(v_x)\in T_{\phi_t(x)}W^s(p)$ since $v_x\in T_xW^s(p)$ - a consequence of the invariance of $W^s(p)$ under the flow $\phi_t$, and $T_{\phi_t(x)}W^s(p)$ converges to $T_pW^s(p)$ as $t \to +\infty$. Moreover $\text{Hess}(f)|_p$ is positive definite on $T_pW^s(p)$
 %  we refer to Thm 8.3.1 of \cite{Jo} 
 , i.e. $\exists ~c>0 ~s.t. ~\text{Hess}(f)|_p(w,w)\geq c|w|^2$ for $\forall w\in T_pW^s(p)$. Thus for suffciently large $t$, $\text{Hess}(f)|_{\phi_t(x)}(\mathrm{d}\phi_t|_x(v_x),\mathrm{d}\phi_t|_x(v_x))  \geq \frac{c}{2}|\mathrm{d}\phi_t|_x(v_x)|^2=cE(t)$. It's basically a linear-algebraic fact: If a quadratic form $B$ is positive definite on a subspace $V$, then small perturbations of $B$ and $V$ still preserve the positivity.

  Now (\ref{ODE}) becomes $\frac{\mathrm{d}}{\mathrm{d}t}E(t) \leq -cE(t), ~\forall t \geq t_0$ for some $t_0$. Integration immediately yields the exponential decay estimate: $E(t)\leq E(t_0)e^{-c(t-t_0)}$.
\end{proof}
\textbf{Remark:} Combining {Theorem}~\ref{shrinkage} with the fact that the union of stable manifolds of all local minima is open dense in $M$, and for each such minimum $q_i$,  $T_xW^s(q_i)=T_xM$, it follows that the flow is shrinking almost everywhere on $M$. It will be used in {Proposition}~3.2.

\section{Applications}
 
   Both {Theorem}~\ref{connectedness} and {Theorem}~\ref{shrinkage} find applications in certain vanishing results. These results generalize the {Proposition}~3.2 and {Proposition}~3.4 in Ni's paper \cite{Ni}, where the Morse function is given by a linear function $l$ of $\mathbb{R}^{m+1}$ restricted to the unit sphere $\Bbb{S}^m$.
   
\begin{proposition}\label{constant}
   If for a smooth map $u:M\to N$ and for some Morse function $f$ on $M$, $du(\nabla f)=0$, then $u$ must be a constant map. 
\end{proposition}
\begin{proof}
    It's a direct consequence of {Theorem}~\ref{connectedness}, since $u$ is constant along the flow lines of negative gradient flow.
\end{proof}

  Another application is the flatness of metric connections on vector bundles. First we introduce the concept of connection Lie derivative: Let $E\to M$ be a vector bundle and let $D$ be a connection on $E$. For a vector field $X$ on $M$ and its associated flow $\phi_t$, we denote by $\tilde\phi_t$ the horizontal lift of $\phi_t$ with respect to the connection $D$, i.e. $\tilde\phi_t$ is the parallel transport along $\phi_t$. The connection Lie derivative on $\omega\in\Omega^r(M,E)$ is defined by
 \begin{equation}
      \mathcal{L}_X^D\omega(Y_1,\cdots,Y_r)=\frac{\mathrm{d}}{\mathrm{d}t}\Bigg |_{t=0}\tilde\phi_{-t}\Big( \omega \big(\mathrm{d}\phi_t(Y_1),\cdots, \mathrm{d}\phi_t(Y_r)   \big) \Big)
  \end{equation}
for $\forall ~Y_1,\cdots, Y_r \in \Gamma(M,TM)$. It has a formula analogous to the normal Lie derivative: 
\begin{equation*}
    \mathcal{L}_X^D\omega(Y_1,\cdots,Y_r)=D_X\big(\omega(Y_1,\cdots,Y_r) \big)-\displaystyle\sum_{i=1}^r\omega(Y_1,\cdots,[X,Y_i],\cdots,Y_r)
\end{equation*}
\begin{proof}
    It's just an argument of Leibniz rule :
\begin{equation*}
    \begin{aligned}
        \mathcal{L}_X^D\omega(Y_1,\cdots,Y_r)(x)=  \lim_{t\to 0} \frac{1}{t} \Big( \tilde\phi_{-t}\big( \omega_{\phi_t(x)} \big(\mathrm{d}\phi_t(Y_1 \big |_x),\cdots, \mathrm{d}\phi_t(Y_r \big |_x)  \big) \big)- \omega_x \big(Y_1 \big |_x,\cdots, Y_r \big |_x) \Big)
    \\ =\lim_{t\to 0} \frac{1}{t} \Big( \tilde\phi_{-t} \big( \omega_{\phi_t(x)}(Y_1\big|_{\phi_t(x)}, \cdots, Y_r\big|_{\phi_t(x)}) \big) - \omega_x \big(Y_1 \big |_x,\cdots, Y_r \big |_x \big) \Big)+ \ \ \ \ \ \ \ \ \ \ \ \ \ \ \ \ \ 
    \\ \sum_{i=1}^r \lim_{t\to 0}  \tilde\phi_{-t}\Big( \omega_{\phi_t(x)} \big( Y_1\big|_{\phi_t(x)} , \cdots,  \mathrm{d}\phi_t(\frac{Y_i|_x-\mathrm{d}\phi_{-t}(Y_i|_{\phi_t(x)})}{t}) ,\cdots, \mathrm{d}\phi_t(Y_r \big |_x) \big) \Big) \ \ \ \ \ \ \ \ \ \ \ \ 
    \\ = D_X\big(\omega(Y_1,\cdots,Y_r) \big)-\displaystyle\sum_{i=1}^r\omega \big(Y_1,\cdots,\mathcal{L}_X(Y_i),\cdots,Y_r \big). \ \ \ \ \ \ \ \ \ \ \ \ \ \ \ \ \ \ \ \ \ \ \ \ \ 
    \end{aligned}
\end{equation*}Notice that it induces the normal formula by considering the trivial connection on the trivial line bundle over $M$.\end{proof}

Recall the induced exterior derivative $d_D:\Omega^{\bullet}(M,E)\to \Omega^{\bullet +1}(M,E)$ of the connection $D$ (see \cite{Ni} for more details):
\begin{equation*}
    \begin{aligned}
        d_D\omega \big(Y_1, \cdots,Y_r \big)=\sum_{i=1}^r (-1)^{i-1}D_{Y_i} \big(\omega(Y_1,\cdots,\hat{Y_i}, \cdots,Y_r) \big)
        \\ +\sum_{1\leq i,j\leq r} (-1)^{i+j} \omega \big([Y_i,Y_j],Y_1,\cdots, \hat{Y_i},\cdots,\hat{Y_j},\cdots,Y_r \big), \ \ \ 
    \end{aligned}
\end{equation*}
we obtain a similar Cartan formula: $\mathcal{L}^D_X=\imath_Xd_D+d_D\imath_X$, which is sufficient to verify for $0$- and $1$-forms, since $\mathcal L_X^D$ is a derivation (of degree 0):
\begin{equation*}
    \begin{aligned}&(\imath_Xd_D+d_D\imath_X)f=d_Df(X)=D_Xf=\mathcal{L}_X^Df; 
    \\ &(\imath_Xd_D+d_D\imath_X)\omega(Y)= d_D\omega(X,Y)+\big(d_D\omega(X)\big)Y 
    \\ &=D_X\omega(Y)-D_Y\omega(X)-\omega([X,Y])+D_Y\omega(X)=\mathcal{L}_X^D\omega(Y).
    \end{aligned}
\end{equation*}

  One final point to mention is the endomorphism bundle $\text{End}(E)$, together with its connection $\tilde D=[d_D,\cdot]$ induced by the connection $D$ on $E$. Note that the curvature operator $R^D\in \Omega^2(M,\text{End}(E))$ and we have the well-known Bianchi identity $d_{\tilde D}R^D=[d_D,d_D^2]=0$.

\begin{proposition}\label{flatness}
    Suppose $E\to M$ is a Riemannian vector bundle endowed with a metric $h$, and $D$ is a connection on $E$ compatible with $h$. If for a Morse function $f$, $\mathcal{L}^{\tilde D}_{\nabla f}R^D=0$, then $R^D=0$, \textnormal{i.e.} $D$ is flat.
\end{proposition}
\begin{proof}
   Again we let $\phi_t$ be the negative gradient flow of $f$. This time we denote by $\tilde\phi_t$ the horizontal lift of $\phi_t$ to $\text{End}(E)$, with respect to the connection $\tilde D$. Since the induced connection $\tilde D$ is also a metric connection for the induced metric $\tilde{h}$ on $\text{End}(E)$, the parallel transport $\tilde\phi_t$ preserves $\tilde{h}$. The condition $\mathcal{L}^{\tilde D}_{\nabla f}R^D=0$ implies the invariance of $R^D$ under the flow $\tilde\phi_t$. 

   Now for any local minimum $q$, and all $ x\in W^s(q)$, the invariance of $R^D$ takes the form 
\begin{equation*}
    R^D_x(u_x,v_x)=\tilde\phi_{-t}\Big|_{\phi_t(x)}R^D_{\phi_t(x)}\big(\mathrm{d}\phi_t|_x(u_x),\mathrm{d}\phi_t|_x(v_x) \big)
\end{equation*}    
for $\forall u_x~,v_x\in T_xM$. Taking the norm on both sides, we obtain 
\begin{equation*}
    \big| R^D_x(u_x,v_x) \big|_{\tilde h}=\big| R^D_{\phi_t(x)}\big(\mathrm{d}\phi_t|_x(u_x),\mathrm{d}\phi_t|_x(v_x) \big) \big|_{\tilde h} 
\end{equation*}
. However, as $t \to+\infty$, the right-hand side decays to $0$ due to the shrinkage of the flow $\phi_t$ (cf. {Remark} below {Theorem}~\ref{shrinkage}). Therefore $R^D_x=0$, and since these points $x$ form an dense subset of $M$, we obtain the flatness of $D$.
\end{proof}
 We conclude this note with two remarks. First, the condition $\mathcal{L}^{\tilde D}_{\nabla f}R^D=d_{\tilde D}\imath_{\nabla f}R^D=0$ is weaker than the condition $\imath_{\nabla f}R^D=0$ in {Proposition}~3.4 of \cite{Ni}, owing to the Cartan formula and Bianchi identity above, and the proof is also different. Second, by viewing $\mathrm{d}u$ as a  $u^{-1}TN$ 1-form $\in \Omega^1(M,u^{-1}TN)$, we could get the same vanishing result in {Proposition}~\ref{constant} with weaker condition:
\begin{proposition}
    If for a smooth map $u:M\to N$ and for some Morse function $f$ on $M$, $\mathcal{L}_{\nabla f}^{\tilde \nabla} du=0$, then $u$ must be a constant map. Here $\tilde \nabla$ is the pull-back connection of the Levi-Civita connection on $N$.
\end{proposition}
  The proof is identical to that of {Proposition}~\ref{flatness}, and the condition is indeed weaker for the same reason: Cartan formula and $\mathrm{d}_{\tilde\nabla} \mathrm{d}u=0$, which appears in the proof of  {Theorem}~\ref{shrinkage}.

\bigskip

%\noindent\textbf{Acknowledgments.} {The author thanks Professor Lei Ni for bringing the refernece \cite{Ni} to his attention, and Professor Jianqing Yu for helpful discussions regarding Morse homology.}

\end{document}